\documentclass[12pt,reqno,twoside]{amsart}
\usepackage{amsmath}
\usepackage{amssymb}
\usepackage{epsfig}
\usepackage{color}
\usepackage{hyperref}

\numberwithin{equation}{section}

\title[DIOPHANTINE EXPONENTS OF AFFINE SUBSPACES]{DIOPHANTINE EXPONENTS OF AFFINE SUBSPACES: THE SIMULTANEOUS APPROXIMATION CASE}

\author{Yuqing Zhang}

\address{Brandeis University, Waltham MA
02454-9110 {\tt yqzhang@brandeis.edu}}

\newif\ifdraft\drafttrue

\draftfalse

\newcommand{\R}{{\mathbb{R}}}

\newcommand{\Z}{{\mathbb{Z}}}

\newcommand{\N}{{\mathbb{N}}}

\newcommand{\q}{{\bf{q}}}

\newcommand{\GL}{\operatorname{GL}}

\newcommand{\SL}{\operatorname{SL}}

\newcommand{\diag}{{\rm diag}}

\newcommand{\ou}{{\omega(\y)}}
\newcommand{\su}{{\sigma(\y)}}

\newcommand{\x}{{\bf x}}

\newcommand{\w}{{\bf w}}

\newcommand{\p}{{\bf p}}
\newcommand{\y}{{\bf y}}

\newcommand {\ignore}[1]  {}

\newtheorem{thm}{Theorem}[section]

\newtheorem{lem}[thm]{Lemma}

\newtheorem{prop}[thm]{Proposition}

\newtheorem{cor}[thm]{Corollary}

\newtheorem{remark}[thm]{Remark}

\begin{document}

\begin{abstract}
We apply  nondivergence estimates for flows on homogeneous spaces to compute
Diophantine exponents of affine subspaces of $\R^n$ and their nondegenerate submanifolds.

\end{abstract}


\maketitle

\section{Introduction}

Given any $\y=(y_1,\ldots,y_n) \in \R^n$ we set its norm as follows:
\begin{equation}
\|\y\|=\max\{|y_1|,|y_2|,\ldots,|y_n|\}
\end{equation}

 We may view $\y$ as a linear form and define its Diophantine
exponent as
\begin{equation}
\label{eq: omega} \omega(\y)=\textrm{sup} \{v|\textrm{ } \exists \infty
\textrm{ many } \q \in \Z^n  \textrm{ with } |\q\y+p|<\|\q\|^{-v}
\textrm{ for some } p \in Z\}
\end{equation}
where $\q\y=q_1y_1+q_2y_2+\ldots+q_ny_n$.

Alternatively we may define Diophantine exponent of  $\y$ in the context of simultaneous approximation:
\begin{equation}
\label{eq: sigma} \sigma(\y)=\textrm{sup} \{v|\textrm{ } \exists \infty
\textrm{ many } q \in \Z \textrm{ with } \|q\y+\p\|<|q|^{-v}
\textrm{ for some } \p \in \Z^n\}
\end{equation}

It can be deduced from Dirichlet's Theorem (\cite{Cassels}) that
\begin{equation}\label{eq: minkowski}
\su\geqslant\frac{1}{n}, \quad \ou\geqslant n \quad \forall \y \in \R^n
\end{equation}

Khintchine's Transference Theorem (see chapter V of \cite{Cassels} for instance) tells
\begin{equation}\label{eq: transference1}
\dfrac{\omega(\y)-n+1}{n}\geqslant \sigma(\y)\geqslant
\dfrac{1}{n-1+n/\omega(\y)},\quad \forall \; \y \in\R^n
\end{equation}
In particular $\su=\frac{1}{n}$ if and only if $\ou=n$.  We call $\y$  not very well approximable if  $\su=\frac{1}{n}$ and call $\y$
 very well approximable otherwise. It is known that
 $\su=\frac{1}{n}$ and $\ou=n$ for $a.\mathcal{}e. \;\y$, hence the set of  not very well approximable vectors has full Lebesgue measure .

 Following \cite{exponent}  the Diophantine exponent $\omega(\mu)$ of a
 Borel measure $\mu$ is set to be the $\mu$-essential supremum of the $\omega$
function, that is,
\begin{equation}
\label{eq: deu} \omega(\mu)=\textrm{sup} \{v|\textrm{ }
\mu\{\y|\textrm{ }\omega(\y)>v\}>0 \}
\end{equation}
If $M$ is a smooth submanifold of $\R^n$ and $\mu$ is the measure class of the Riemannian volume on $M$ (more precisely put, $\mu$ is
the pushforward $\mathbf{f}_*\lambda$ of $\lambda$ by any smooth map $\mathbf{f}$ parameterizing $M$), then the Diophantine exponent
of $M$, $\omega(M)$, is set to be equal to $\omega(\mu)$. In the spirit of \eqref{eq: deu} let us define
\begin{equation}
\label{eq: deusigma} \sigma(M)=\sigma(\mu)\stackrel{def}{=}\textrm{sup} \{v|\textrm{ }
\mu\{\y|\textrm{ }\sigma(\y)>v\}>0 \}
\end{equation}
$\omega(M)\geqslant n$ and $\sigma(M)\geqslant\dfrac{1}{n}$ by Dirichlet's Theorem combined with \eqref{eq: deu} and \eqref{eq: deusigma}.
M is called extremal if  $\sigma(M)=\frac{1}{n}$ or $\omega(M)=n$.
A trivial example of an extremal
submanifold of $\R^n$ is $\R^n$ itself.

 K. Mahler (\cite{M}) conjectured in 1932 that
\begin{equation}\label{eq: curve}
 M=\{(x,x^2,\ldots,x^n)|x \in \R\}
\end{equation}
 is an extremal submanifold. This was proved by
Sprind\u{z}uk (\cite{Sp1}) in 1964. The curve of \eqref{eq: curve} has a notable property that it does not lie in
 any affine subspace of $\R^n$. We might describe and formalize this property in terms of nondegeneracy condition as follows.
 Let \\$\mathbf{f}=(f_1,\ldots,f_n): U\rightarrow \R^n$ be a differentiable map where $U$
is an open subset of $\R^d$. $\mathbf{f}$ is called nondegenerate in
an affine subspace $L$ of $\R^n$ at $\x \in U$ if $\mathbf{f}(U)
\subset L$ and the span of all the partial derivatives of
$\mathbf{f}$ at $\x$ up to some order coincides with the linear part
of $L$. If $M$ is a $d$ dimensional submanifold  of $L$ we will say
that $M$ is nondegenerate in $L$ at $\y \in M$ if any diffeomorphism
of $\mathbf{f}$ between an open subset $U$ of $\R^d$ and a
neighborhood of $\y$ in $M$ is nondegenerate in $L$ at
$\mathbf{f}^{-1}(\y)$. We will say $M$ is nondegenerate in $L$ if it
is nondegenerate in $L$ at almost all points of $M$.

   It was conjectured by Sprind\u{z}uk (\cite{Sp2})
in 1980 that almost all points on a nondegenerate analytic submanifold of $\R^n$ are not very well approximable.
In 1998 D. Kleinbock and  G.A. Margulis   proved that
\begin{thm}([KM])
Let $M$ be a smooth nondegenerate submanifold of $R^n$, then $M$ is extremal, i.e. almost all points of $M$ are not very well approximable.
\end{thm}

\cite{extremal} studies the conditions under which an affine subspace is  extremal and
showed that an affine space is extremal if and only if its nondegenerate submanifolds are extremal.
\cite{exponent} derives formulas for  computing $\omega(L)$ and $\omega(M)$ when $L$ is not extremal and $M$ is an arbitrary nondegenerate
submanifold in it. This breakthrough is achieved through sharpening of some nondivergence estimates in the space of unimodular lattices (see Lemma \ref{lem:  thm2.2} and Lemma \ref{lem:  thm2.2weaker} for review ). \cite{exponent} proves that
\begin{thm}\label{thm: exponent}[Theorem 0.3 of \cite{exponent}]
If $L$ is an affine subspace of  $\R^n$ and $M$ is a non-degenerate
submanifold in $L$, then
\begin{equation}
  \textrm{ }\omega(M)=\omega(L)=\inf  \{\omega(\x)|\textrm{ }\x \in L\}=\inf \{\omega(\x)|\textrm{ }\x \in M\}
\end{equation}
\end{thm}

This paper goes on to compute Diophantine exponents of nonextremal subspaces in the $\sigma$ context. We follow the strategy of associating Diophantine property of vectors with behavior of certain trajectories in the space of lattices.
Combined with dynamics we use nondivergence estimates in its strengthened format (Lemma \ref{lem:  thm2.2}) to  prove the following:
\begin{thm}\label{thm: mainthm}
If $L$ is an affine subspace of  $\R^n$ and $M$ is a non-degenerate
submanifold in $L$, then
\begin{equation}
  \textrm{ }\sigma(M)=\sigma(L)=\inf  \{\sigma(\x)|\textrm{ }\x \in L\}=\inf \{\sigma(\x)|\textrm{ }\x \in M\}
\end{equation}
\end{thm}
Theorem \ref{thm: mainthm} shows that  simultaneous Diophantine exponents of affine subspaces are inherited by their nondegenerate
submanifolds. Though Theorem \ref{thm: exponent} and \ref{thm: mainthm} look much alike, the latter cannot be deduced directly from the former. A simplified account for this
can be found in \eqref{eq: transference1}. When $\omega(\y)>n$, $\frac{\omega(\y)-n+1}{n}>
\frac{1}{n-1+n/\omega(\y)}$ and $\omega(\y)$ might take on any value between the two fractions ( we refer readers to
\cite{Jarnik} for such examples).

We will also compute explicitly Diophantine exponents of affine subspaces in terms of the coefficients of their
parameterizing maps. One instance of our accomplishment is the derivation of $\sigma(L)$ where $L$ is a hyperplane:
 Consider   $L \subset \R^n$ parameterized by
\begin{equation}\label{eq: ldefi}
(x_1,x_2,\ldots,x_{n-1})\rightarrow
(a_1x_1+\ldots+a_{n-1}x_{n-1}+a_n,x_1,\ldots,x_{n-1})
\end{equation}

If we denote the vector $(a_1,\ldots,a_n)$ by $\mathbf{a}$, then in \S4 we will establish
\begin{thm}\label{thm: hyperplane}
For L as described in \eqref{eq: ldefi}
\begin{equation}
 \sigma(L)=\max
\{1/n,\dfrac{\omega(\mathbf{a})}{n+(n-1)\omega(\mathbf{a})}\}
\end{equation}
\end{thm}

The main result of this paper is actually much more general than Theorem \ref{thm: mainthm}. We will be considering maps from Besicovitch metric
spaces endowed with Federer measures (we postpone definitions of terminology till \S2). We will be able to include in our results
measures of the form $\mathbf{f}_*\mu$ where $\mu$ satisfies certain decay conditions as discussed in \cite{KLW}.

In \S4 we will also study examples where $\sigma(L)$ is determined by the coefficients of its
parameterizing map in a more intricate manner. In \S5 we will give an illustration as to how the process of ascertaining
$\sigma(L)$ differs from that of ascertaining  $\omega(L)$.

\bigskip

\section{Quantitative nondivergence}

We will study homogeneous dynamics and how these relate to Diophantine approximation of vectors. First we define the space of unimodular lattices as follows:
\begin{equation}
\Omega_{n+1}\stackrel{\mathbf{def}}{=}\SL(n+1,\R)\diagup \SL(n+1,\Z)
\end{equation}
 $\Omega_{n+1}$ is non-compact,   and can be decomposed as
\begin{equation} \label{eq: space}
\Omega_{n+1}=\bigcup_{\epsilon>0}K_\epsilon
\end{equation}
where
\begin{equation}
K_\epsilon=\{\Lambda \in \Omega_{n+1}| \textrm{  }\|v\| \geqslant
\epsilon \textrm{  for all nonzero   }v \in \Lambda\}
\end{equation}
Each $K_\epsilon$ is compact by Mahler's compactness criterion (see \cite{M}).
\begin{remark}
$\|\;\;\|$ can be either the maximum  or Euclidean norm on $\R^{n+1}$ and both can be used for decomposing $\Omega_{n+1}$
into union of compact subspaces because for each $v=(v_1,\ldots,v_{n+1})$ there exists $C_1>0$ and $C_2>0$ such that
$C_1\max\{|v_1|,\ldots,|v_{n+1}|\}\leqslant  \sqrt{v_1^2+\ldots+v_{n+1}^2}$\\$\leqslant C_2\max\{|v_1|,\ldots,|v_{n+1}|\}$.
We assume it to be the maximum here and extend to the space of discrete subgroups of $\R^{n+1}$.
For nonzero $\Gamma$ we let $\|\Gamma\|$ be the volume of the quotient space $\Gamma_\R \diagup \Gamma$, where $\Gamma_\R$ is the $\R$ linear
span of $\Gamma$. If $\Gamma=\{0\}$, we set $\|\Gamma\|=1$.
\end{remark}
Next we set
\begin{displaymath}
\Sigma_v\stackrel{\mathbf{def}}{=}\{\y \in \R^n \textrm{  }|\;\exists \infty \textrm{many } q
\in \Z \textrm{ such that }\|q\y-\p\|<|q|^{-v} \}
\end{displaymath}
Obviously $\sigma(\y)=\sup\{v\;|\;\y \in \Sigma_v\}$.

Set $g_t=\diag\{\underbrace{e^{t/n},e^{t/n},\ldots,e^{t/n}}_n,e^{-t}\} \in \SL(n+1,\R)$ with $t\geqslant 0$ and associate
$\y \in \R^n$ with matrix
\begin{equation}\label{eq: defofuy}
u_\y=\left(\begin{array}{cc} I_n & \y\\
0 & 1\end{array}\right)
\end{equation}
Consider lattice
\begin{equation}\label{eq: uyzn}
\left\{ \left(
          \begin{array}{c}
            q\y+\p \\
            q \\
          \end{array}
        \right)\Big|\;q \in \Z, \p \in \Z^n
\right\}=u_{\y}\Z^{n+1}
\end{equation}
When we have $g_t$ act on  vectors in $u_{\y}\Z^{n+1}$ as defined by \eqref{eq: uyzn}, the first $n$ components will be expanded and the last one ($q$) will be contracted. A definitive correlation between $\sigma(\y)$ and trajectory of certain lattices in $\Omega_{n+1}$ was
proposed and proved in \cite{extremal}. This is a special case of Theorem 8.5 of \cite{KM3} on logarithm laws.
\begin{lem}\label{lem: abequi}
Suppose we are given a set $E \in \R^2$ which is discrete and
homogeneous with respect to positive integers, and take $a,b>0$,
$v>a/b$. Define $c$ by $c=\dfrac{bv-a}{v+1}$, then the following are
equivalent:

\begin{enumerate}
\item $\exists (x,z) \in E$  with arbitrarily large $|z|$ such that
$|x| \leqslant |z|^{-v}$
\item $\exists$  arbitrarily large $t>0$ such
that for some $(x,z) \in E$ one has\\
max$(e^{at}|x|,e^{-bt}|z|) \leqslant e^{-ct}$
\end{enumerate}

\end{lem}

In the light of  Lemma \ref{lem: abequi}, if we set $v>1/n$, $\y \in \R^n$ and \\$E=\{(\|q\y+\p\|,|q| )\textrm{  }
|\textrm{  }q \in \Z, \p \in \Z^n\}$, (1) of Lemma \ref{lem: abequi} is equivalent to
\[\sigma(\y) \y \in \Sigma_v\]
By setting $a=1/n, b=1$  and $\R_+=\{x\in \R|x\geqslant 0 \}$ one sees (2) of Lemma \ref{lem: abequi} is
equivalent to
\begin{equation} \label{eq: gt}
g_tu_y\Z^ {n+1} \notin
K_{e^{-ct}} \textrm{ for an unbounded set of }t \in \R_+
\end{equation}
where $\|\;\;\|$ is the maximum  norm and
\begin{equation} \label{eq: cv}
c=\dfrac{v-1/n}{v+1}
\Leftrightarrow  v=\dfrac{1/n+c}{1-c}=\dfrac{1+nc}{n(1-c)}
\end{equation}

If, in compliance with the definition of $\su$, we set
\begin{equation}\label{eq: gammay}
\gamma(\y)= \textrm{sup} \{c \;| \textrm{  }g_tu_\y\Z^ {n+1} \notin
K_{e^{-ct}} \textrm{ for an unbounded set of }t \in \R_+\}
\end{equation}
then by \eqref{eq: cv}  we have
\begin{equation}\label{eq: sigmay}
\sigma(\y)=\dfrac{1+n\gamma(\y)}{n(1-\gamma(\y))}
\end{equation}

\medskip

Suppose $\nu$ is a measure on $\R^n$,  and $v\geqslant 1/n$ , by \eqref{eq: deu} and what ensues
$\sigma(\nu)\leqslant v$ if and only if
\begin{equation}\label{eq: lowernu}
\nu(\Sigma_u)=0 \; \; \forall u>v.
\end{equation}
\eqref{eq: lowernu} is  equivalent to
\begin{equation}
\label{eq: equilowernu}
\nu(\{\y|\textrm{  }g_tu_\y\Z^ {n+1} \notin K_{e^{-dt}} \textrm{
for an unbounded set of }t \in \R_+\})=0, \; \forall d>c
\end{equation}
where $c$ is related to $v$ via fractions of \eqref{eq: cv}.

\eqref{eq: equilowernu} can be further simplified into
\begin{equation}
\label{eq: integerequilowernu}
\nu(\{\y|\textrm{  }g_tu_\y\Z^ {n+1} \notin K_{e^{-dt}} \textrm{
for an unbounded set of }t \in \N\})=0, \; \forall d>c
\end{equation}

By the Borel-Cantelli Lemma, a sufficient condition for $\sigma(\nu)\leqslant
v$, or \eqref{eq: integerequilowernu} is
\begin{equation}\label{eq: bcl}
\sum_{t=1}^\infty\nu(\{\y|\textrm{  }g_tu_\y\Z^ {n+1} \notin
K_{e^{-dt}}\}<\infty\textrm{,      }\forall d>c
\end{equation}

The following lemma, established in \cite{exponent}, serves as a sharpening of
quantitative nondivergence.  First an assembly of relevant concepts from the same resource (to trace their historical development
see also \cite{KM}, \cite{KLW})

A metric space $X$ is called $N-Besicovitch$ if for any bounded
subset $A$ and any family $\beta$ of nonempty open balls of $X$ such
that each $x \in A$ is a center of some ball of $\beta$, there is a
finite or countable subfamily $\{\beta_i\}$ of $\beta$ covering $A$
with multiplicity at most $N$. $X$ is $Besicovitch$ if it is $N-Besicovitch$ for some $N$.
\medskip

Let $\mu$ be a locally finite Borel measure on $X$, $U$ an open
subset of $X$ with $\mu(U)>0$. Following \cite{KLW} we call $\mu$  $D-Federer$ on $U$ if
\begin{displaymath}
\sup _{\begin{subarray}{1}x \in \textrm{supp }\mu,\; r>0
\\ B(x,3r)\subset U  \end{subarray}}\dfrac{\mu(B(x,3r))}{\mu(B(x,r)}<D
\end{displaymath}
$\mu$ is said to be $Federer$ if for $\mu$-a.e. $x \in X$ there
exists a neighborhood $U$ of $x$ and $D>0$ such that $\mu$ is
$D-Federer$ on $U$.

An important illustration of the above notions is that $\R^d$ is Besicovitch and $\lambda$, the Lebesgue measure
is $Federer$. Many natural measures supported on fractals are also known to be $Federer$ (see \cite{exponent} for technical details).

\medskip
For a subset $B$ of $X$ and a function $f$ from $B$ to a normed space with norm $\|\textrm{ } \|$, we define
  $\|f\|_B= \textrm{sup }_{x \in B}\|f(x)\|$. If $\mu$ is a locally finite
  Borel measure on $X$ and $B$ a subset of $X$ with $\mu(B)>0$
  $\|f\|_{\mu, B}$ is set to be $\|f\|_{B \cap \textrm{supp }\;\mu}$.

 A function $f: X\rightarrow \R$ is called $(C,\alpha)$-good on $U \subset X$
with respect to $\mu$ if for any open ball $B$ centered in supp
$\mu$ one has
\begin{displaymath}
\forall \varepsilon>0 \qquad  \mu(\{x \in B\;|\;|f(x)|<\varepsilon\})\leqslant
C\big(\dfrac{\varepsilon}{\|f\|_{\mu,B}}\big)^{\alpha}\mu(B).
\end{displaymath}
Roughly speaking a function is $(C,\alpha)$-good if the set of points where it takes small value
has small measure. In Lemma \ref{lem: thm2.2} we will see that functions of the form $\x\rightarrow
\|h(\x)\Gamma\|$ , where $\Gamma$ runs through subgroups of $\Z^{n+1}$, are $(C, \alpha)$-good with uniform $C$ and $\alpha$.

Let $\mathbf{f}=(f_1,\ldots,f_n)$ be a map from $X$ to $\R^n$. Following \cite{exponent} we say that ($\mathbf{f}, \mu$) is good
at $x \in X$ if there exists a neighborhood V of $x$ such that any linear combination of $1,f_1,\ldots,f_n$
is $(C,\alpha)$-good on V with respect to $\mu$ and ($\mathbf{f}, \mu$) is good if ($\mathbf{f}, \mu$) is good at $\mu$-almost
every point. Reference to measure will be omitted if $\mu=\lambda$,  and we will simply say that $\mathbf{f}$
is good or good at $x$. For example polynomial maps are good. \cite{extremal} proved the following result:
\begin{lem}\label{lem: goodmap}
Let $L$ be an affine subspace of $\R^n$ and let $\mathbf{f}$ be a smooth map from $U$, an open subset of $\R^d$ to $L$
which is nondegenerate at $\x \in U$, then $\mathbf{f}$ is good at $\x$.
\end{lem}
Furthermore if $L$ is an affine subspace of $\R^n$ and $\mathbf{f}$
a map from X into $L$, following \cite{exponent} we say ($\mathbf{f}, \mu$) is nonplanar in $L$ at $x
\in \textrm{supp }\mu$ if $L$ is equal to the intersection of all
affine subspaces containing $\mathbf{f}(B\cap \textrm{supp }\mu)$ for some open neighborhood $B$ of $x$.
($\mathbf{f}, \mu$) is nonplanar in $L$ if ($\mathbf{f}, \mu$) is nonplanar in $L$ at $\mu$-a.e. $x$. We skip saying $\mu$ when $\mu=\lambda$ and
skip $L$ if $L=\R^n$. From definition
($\mathbf{f}, \mu$) is nonplanar if and only if for any open $B$ of positive measure, the restrictions of $1,f_1,\ldots,f_n$ to $B\bigcap \textrm{supp }\mu$ are linearly independent over $\R$. Clearly nondegeneracy in $L$ implies nonplanarity in $L$.  Nondegenerate smooth maps from $\R^d$ to
$\R^n$ as in Lemma \ref{lem: goodmap} give typical examples of nonplanarity.

\medskip

Let $\Gamma$ be any discrete subgroup of $\R^{k}$  we denote by $rk(\Gamma)$ the rank of $\Gamma$ when viewed as a $\Z$-module.
We denote by $S_{n+1, j}$  the set
of subgroups of order j in $\Z^{n+1}$ for $1\leqslant j\leqslant n+1$.

\begin{lem}\label{lem:  thm2.2}
Let $k$, $N \in \N$ and $C,D,\alpha,\rho >0$ and suppose we are
given an $N$-Besicovitch metric space $X$, a ball $B=B(x_0,
r_0)\subset X$, a measure $\mu$ which is $D$-Federer on
$\tilde{B}=B(x_0, 3^kr_0)$ and a map $h$: $\tilde{B}\rightarrow
\GL_k(\R)$. Assume the following two conditions hold:
\begin{enumerate}
\item $\forall \;\Gamma \subset \Z^k$,  the function $ x\rightarrow
\|h(x)\Gamma\|$ is $(C, \alpha)$-good on $\tilde{B}$ with respect to
$\mu$;

\item $\forall \; \Gamma \subset \Z^k$, $\|h(\cdot )\Gamma\|_{\mu, B}\geqslant \rho^{rk(\Gamma)}$
\end{enumerate}
Then for any positive $\epsilon \leqslant \rho$ one has
\begin{equation}
\mu(\{x \in B|\textrm{ } h(x)\Z^k \notin K_{\epsilon}\})\leqslant
kC(ND^2)^k(\dfrac{\epsilon}{\rho})^{\alpha}\,\mu(B)
\end{equation}
\end{lem}

Historically  one theorem of \cite{KM} established the above lemma in its weaker form:
\begin{lem}\label{lem:  thm2.2weaker}
Let $k$, $N \in \N$ and $C,D,\alpha,\rho >0$ and suppose we are
given an $N$-Besicovitch metric space $X$, a ball $B=B(x_0,
r_0)\subset X$, a measure $\mu$ which is $D$-Federer on
$\tilde{B}=B(x_0, 3^kr_0)$ and a map $h$: $\tilde{B}\rightarrow
\GL_k(\R)$. Assume the following two conditions hold:
\begin{enumerate}
\item $\forall \;\Gamma \subset \Z^k$,  the function $ x\rightarrow
\|h(x)\Gamma\|$ is $(C, \alpha)$-good on $\tilde{B}$ with respect to
$\mu$;

\item $\forall \; \Gamma \subset \Z^k$, $\|h(\cdot )\Gamma\|_{\mu, B}\geqslant \rho$
\end{enumerate}
Then for any positive $\epsilon \leqslant \rho$ one has
\begin{equation}
\mu(\{x \in B|\textrm{ } h(x)\Z^k \notin K_{\epsilon}\})\leqslant
kC(ND^2)^k(\dfrac{\epsilon}{\rho})^{\alpha}\,\mu(B)
\end{equation}
\end{lem}
For the wide number-theoretic applications of Lemma \ref{lem:  thm2.2weaker} we refer readers to papers like
\cite{KM},\cite{KLW}, to name a few. \cite{exponent} proves with an inductive process
that one can replace the second condition
$\|h(\cdot )\Gamma\|_{\mu, B}\geqslant \rho$ of Lemma \ref{lem:  thm2.2weaker} with $\|h(\cdot )\Gamma\|_{\mu, B}\geqslant \rho^{rk(\Gamma)}$
and thus obtains an strengthening of nondivergence estimates as recorded in  Lemma \ref{lem:  thm2.2}. Both \cite{exponent} and
the present paper exploit Lemma \ref{lem: thm2.2} to get Diophantine exponents of non-extremal spaces.

\begin{prop}\label{prop: proposition}
Let $X$ be a Besicovitch metric space, $B=B(x, r)\subset X$, $\mu$ a
measure which is D-Federer on $\tilde{B}=B(x, 3^{n+1}r)$ for some
$D>0$ and $ \mathbf{f}$ a continuous map from $\tilde{B}$ to $\R^n$.
Take $c \geqslant 0$ and assume that
\begin{enumerate}
\item
$\exists C,\alpha >0$ such that all the functions   $x\rightarrow
\|g_tu_{\mathbf{f}(x)}\Gamma\|$, $\Gamma \subset \Z^{n+1}$ are $(C,
\alpha)$- good on  $\tilde{B}$ with respect to $\mu$

\item
for any $d>c$, $\exists T=T(d)>0$ such that for any $t\geqslant T$
and
 any $\Gamma \subset \Z^{n+1}$ one has
\begin{equation}\label{eq: critical}
\|g_tu_{\mathbf{f}(\cdot)}\Gamma\|_{\mu,
 B}\geqslant e^{-rk(\Gamma)dt}
\end{equation}

\end{enumerate}
Then $\sigma(\mathbf{f}_*(\mu|_B))\leqslant v$ , where
$v=\dfrac{1/n+c}{1-c}$.

\end{prop}

\begin{proof}

Apply Lemma \ref{lem: thm2.2} with $k=n+1$, $\mu=\mathbf{f}_*(\mu|_B)$, $h(x)=g_tu_{\mathbf{f}(x)}$
$\rho=e^{-ct}$ and $\epsilon=e^{-dt}$. $d\geqslant c\Leftrightarrow \epsilon\leqslant \rho$. It follows that
\begin{equation}
\mu(\{x \in B |\textrm{  }h(x)\Z^ {n+1} \notin
K_{e^{-dt}}\}
\leqslant const \cdot e^{-\alpha\frac{d-c}{2}t}\,\mu(B)\;\forall \;t\geqslant T
\end{equation}
Hence
\begin{equation*}
\sum_{t=1}^{\infty}\mu(\{x \in B |\textrm{  }h(x)\Z^ {n+1} \notin
K_{e^{-dt}}\}<\infty \quad \forall \;d>c
\end{equation*}

By previous discussion concerning \eqref{eq: bcl}, we conclude that  \\ $\sigma(\mathbf{f}_*(\mu|_B))\leqslant v$ for
$v=\dfrac{1/n+c}{1-c}$, as desired.

\end{proof}
\medskip

To get an appreciation of the purport of the proposition, let us
turn to the consequences of one of the conditions failing to be met.

\begin{lem}\label{lem:  negation}
Let $\mu$ be a measure on a set $B \subset \R^n$, take $c>0$, $v>1/n$ and
$c=\frac{v-1/n}{v+1}$. Let $\mathbf{f}$ be a map from $B$ to
$\R^n$ such that \eqref{eq: critical} does not hold, then
\begin{equation}
\mathbf{f}(B\bigcap \mathrm{supp} \textrm{ } \mu)\subset \Sigma_u
\textrm{ for some }u>v
\end{equation}
\end{lem}

\begin{proof}

If  \eqref{eq: critical} does not hold, $\exists j$ with $ 1\leqslant j \leqslant n+1$ , a
sequence $t_i\rightarrow \infty$ and a sequence of discrete subgroups $\Gamma_i \in S_{n+1, j}$ such that for some
$d>c$
\begin{equation}
\forall x \in B\bigcap \mathrm{supp} \textrm{ } \mu \quad
\|g_tu_{\mathbf{f}(x)}\Gamma_i\|<e^{-jdt_i}
\end{equation}
 By Minkowski's lemma,
we have $\forall i$, $\forall x\in B\bigcap \mathrm{supp} \textrm{ } \mu$
there exists nonzero vector $v \in g_{t_i}u_{\mathbf{f}(x)}\Gamma_i
$ with $\|v\|\leqslant 2^je^{-dt_i}$
therefore
\begin{equation}
g_{t_i}u_{\mathbf{f}(x)}\Z^{n+1} \notin
K_{2^je^{-dt_i}} \; \textrm{  for an undounded set of }t
\end{equation}
 Hence  $ \gamma(\mathbf{f}(x))\geqslant d$ by \eqref{eq: gammay} and $ \sigma(\mathbf{f}(x))\geqslant u$ for some $u>v$
by \eqref{eq: sigmay}
\end{proof}
\bigskip

\section{Applications and calculations}
In this part we will utilize the theories established in \S 2 to get
some tangible applications.

Let $L$ be an $s$-dimensional affine subspace of $\R^n$.  Throughout  we will parameterize it as
\begin{equation} \label{eq: par}
\x \to (\tilde{\x}A,\x)
\end{equation}
where $\tilde{\x}$ stands for $( \x, 1)$, $\x \in \R^s$ and $A$
$\in M_{s+1,n-s}$ where  $M_{s+1,n-s}$ denotes the set of matrices of dimension $(s+1) \times (n-s)$.

We record the following observation:
\begin{prop}
Let $L$ be an $s$-dimensional affine subspace of $\R^n$ described by \eqref{eq: par}, then
\begin{equation}
\frac{1}{n}\leqslant \sigma(L)\leqslant \frac{1}{s}
\end{equation}
\end{prop}
\begin{proof}
Note that $\sigma(\y) \geqslant\frac{1}{n}$ for all $\y \in L$ hence $\sigma(L)\geqslant \frac{1}{n}$.\\
Also by \eqref{eq: par} for all $\y \in L$, $\sigma(\y)\leqslant \sigma(x_1,\ldots,x_s)$, hence\\
 $\sigma(L)\leqslant \sigma(\R^s)= \frac{1}{s}$.
\end{proof}

Although for any particular $\y \in L$, $\su$ is determined by how $L$ is  parameterized , later development will show that
$\sigma(L)$ is independent of parametrization. In brief, we are merely interested in whether a set is null or not, and that
is unaltered under invertible linear transformations.

For any matrix $A \in M_{s+1,n-s}$ we define
\begin{equation}
\label{eq: matrixomega} \omega(A)=\textrm{sup} \{v|\textrm{ } \exists \infty
\textrm{ many } \q \in \Z^{n-s}  \textrm{ with } \|A\q+\p\|<\|\q\|^{-v}
\textrm{ for some } \p \in \Z^{s+1}\}
\end{equation}

Comparing \eqref{eq: matrixomega} with \eqref{eq: sigma} and \eqref{eq: omega}, we see that given vector $\y=(y_1,\ldots,y_n)$
\begin{equation}
\omega(A)=\ou \;\; if\;A=\y,\qquad \omega(A)=\su\; \; if\;A=\y^T
\end{equation}
\medskip

 Suppose $\R^{n+1}$ has standard basis $
\mathbf{e}_1,\dots,\mathbf{e}_{n+1}$, and if we extend the Euclidean
structure of $\R^{n+1}$ to $\bigwedge^j(\R^{n+1})=\bigotimes ^j(\R^{n+1})\backslash W_j$
where $W_j$ is the subspace of $j-$tensors generated by transposition,
then for all
\begin{equation*}
I=\{i_1, i_2,\dots,
i_j\}\subset\{1,2,\dots, n+1\}, \quad i_1<i_2<\dots<i_j
\end{equation*}
$\{\mathbf{e}_I\;|\; \mathbf{e}_I=\mathbf{e}_{i_1}\wedge
\mathbf{e}_{i_2} \wedge \dots \wedge \mathbf{e}_{i_j}, \;
\textrm{\#}I=j\}$ form an orthogonal basis of\\
$\bigwedge^j(\R^{n+1})$.

If  a discrete subgroup $\Gamma \subset \R^{n+1}$ of rank $j$ is
viewed as a $\Z$-module with basis
$\mathbf{v}_1,\ldots,\mathbf{v}_j$ then we may represent it by  exterior product\\
$\w=\mathbf{v}_1\wedge\ldots\wedge \mathbf{v}_j$. Observing
$\|\Gamma\|=\|\w\|$,  we will be able to compute
$\|g_tu_\mathbf{f}\Gamma\|_{\mu,B}$ as in  \eqref{eq: critical}  directly.

 Further computation shows (up to $\pm$ signs of permutations)
 \begin{equation}
 u_y\mathbf{e}_i=  \left\{ \begin{array}{ll}
\mathbf{e}_i & \textrm{if $i \neq n+1$}\\
\\
\sum_{i=1}^n {y_j\mathbf{e}_j} +\mathbf{e}_{n+1}  & \textrm{if $i=
n+1$}
\end{array}\right.
 \end{equation}

Hence according to properties of exterior algebra,

\begin{equation}
 u_y\mathbf{e}_I=  \left\{ \begin{array}{ll}
\mathbf{e}_I & \textrm{if $n+1 \notin I$}\\
\\
\sum_{i=1}^n {y_j\mathbf{e}_{I\backslash{\{n+1\}}\cup {i}}
+\mathbf{e}_I} & \textrm{if $n+1 \in I $}
\end{array}\right.
 \end{equation}

 Therefore $\w \in \bigwedge ^j(\R^{n+1}) $ under left multiplication
 of $u_\y$ results in
\begin{equation}
u_\y \w=\pi (\w)+\sum_{i=1}^{n+1}C_i(\w)y_i
\end{equation}

where
\begin{equation}
\pi(\w)=\sum_{\substack{ \textrm{\#}I=j \\ {n+1} \in I
}}\langle\mathbf{e}_I, \w\rangle\mathbf{e}_I \qquad
\end{equation}
\begin{equation*}
 C_i(\w)=\sum_{\substack{ i \in I \\ I \subset
\{{1,2,\cdots,n}\}}}\langle\mathbf{e}_{I \backslash{\{i\}} \cup
{\{n+1\}}}, \w\rangle\mathbf{e}_I  \qquad 1\leqslant i\leqslant n
\end{equation*}
\begin{equation}
 C_{n+1}(\w)=\sum_{ I \subset
\{{1,2,\cdots,n}\}}\langle\mathbf{e}_{I}, \w\rangle\mathbf{e}_I,\quad y_{n+1}=1
\end{equation}

Note that $\sum_{i=1}^{n+1}C_i(\w)y_i$ denotes the image of $u_\y
\w$ under the projection from $\bigwedge ^{j}(\R ^{n+1})$ to
$\bigwedge ^{j}(V)$, where $V$ is the space spanned by
$\{\mathbf{e}_1,\mathbf{e}_2, \cdots, \mathbf{e}_n\}$. Apparently
$\bigwedge ^{j}(V)$ is orthogonal to $\pi(\w)$.

\begin{equation}\label{eq: gtuyw}
g_tu_\y \w=e^{-\frac{n+1-j}{n}}\pi
(\w)+e^{\frac{jt}{n}}\sum_{i=1}^{n+1}C_i(\w)y_i
\end{equation}

\eqref{eq: gtuyw} shows that $g_t$ action tends to contract the $\pi(\w)$ part
while extracting its orthogonal complement. As for the norm, up to
some constant,
\begin{equation}
\|g_tu_{\tilde{f}} \w\|=\mathrm{max}(e^{-\frac{n+1-j}{n}}\|\pi
(\w)\|,e^{\frac{jt}{n}}\|\tilde{f}(\textrm{ })C(\w)\|)
\end{equation}
where $\tilde{f}=(f_1, \ldots, f_n,1)$, and
\begin{displaymath}
C(\w)= \left( \begin{array}{ll}
C_1(\w)  \\
C_2(\w)   \\
\vdots  \\
C_{n+1}(\w)
\end{array}  \right)
\end{displaymath}

Denote by $\Theta_{\mu,B}$ the $\R$-linear span of the restriction
of $(f_1, \ldots, f_n,1)$ to $B \bigcap \textrm{supp }\mu$. Suppose $\Theta_{\mu,B}$ has dimension $s+1$.
Let $\mathbf{g}=(g_1,\ldots,g_s,1)$ be a basis of
the above space, then $\exists R \in M_{s+1,n+1}$ such that
$\tilde{f}=\mathbf{g}R$. $\|\tilde{f}C(\w)\|=\|\mathbf{g}RC(\w)\|$.
As the elements of $\mathbf{g}$ are independent, up to some
constant
\begin{equation*}
\|\tilde{f}C(\w)\|=\|RC(\w)\|
\end{equation*}

\eqref{eq: critical} is equivalent to \\
$\forall d>c,\exists T$ such that $\forall t \geqslant T,\forall
j=1,\ldots,n+1$ and $\forall \w \in S_{n+1,j}$ one has
\begin{equation}\label{eq: equicritical}
\max\Big(e^{-\frac{n+1-j}{n}}\|\pi
(\w)\|,\;e^{\frac{jt}{n}}\|RC(\w)\|\Big)\geqslant e^{-jdt}
\end{equation}

We may restate \eqref{eq: equicritical} in the language of Lemma \ref{lem: abequi} in the following manner

Set $E=\big\{(\|RC(\w)\|,\|\pi(\w)\|)\;|\;\w \in S_{n+1,j}\big\}$ which is discrete and homogeneous with respect to
positive integers.

Set  $a=\dfrac{j}{n},\;b=\dfrac{n+1-j}{n}$,  then \eqref{eq: equicritical} means $\forall c>c_0=j\dfrac{v-n}{v+1}$
 the second assumption of Lemma \ref{lem: abequi} does not hold for large enough  $\|\pi(\w)\|$.

This, by the same lemma, is equivalent to the first assumption
not being met with $v$ replaced by any number greater than
\begin{equation}
\dfrac{a+c_0}{b-c_0}=\dfrac{jv}{v+1-jv}
\end{equation}

Therefore \eqref{eq: equicritical} becomes equivalent to \\
$\forall j=1,2,\ldots,n$, $\forall \;u>\frac{jv}{v+1-jv}$ and $\forall \;
\w \in S_{n+1,j}$ with large enough $\|\pi(\w)\|$ one has
\begin{equation}\label{eq: another}
\|RC(\w)\|>\|\pi(\w)\|^{-u}
\end{equation}
\medskip

Now we will  adopt strategies devised in \cite{exponent} to prove Theorem
\ref{thm: mainthm}. Let $L=\tilde{f}(B\bigcap \textrm{supp }\mu)$. Suppose $L$ has dimension
$s$ and $\mathbf{h}:\R^s\rightarrow L$  is an affine isomorphism then
$\exists R \in M_{s+1,n+1}$ such that
\begin{equation}\label{eq: rdefinition}
(h_1,h_2,\ldots,h_n,1)(x)=(x_1,x_2,\ldots,x_s,1)R, \;\forall \x \in
\R^s
\end{equation}

\begin{thm}\label{thm: anothermain}
Let $\mu$ be a Federer measure on a Besicovitch metric space X, L an
affine subspace of $\R^n$, and let $f: X\rightarrow L$ be a
continuous map which is ($f,\mu$)-good  and $(f,\mu)$- nonplanar in $L$. Then
the following statements are equivalent for $v\geqslant 1/n$:

\begin{enumerate}
\item $\{x \in\textrm{supp }\mu|f(x) \notin \Sigma_u\}$ is
nonempty for any $u>v$

\item $\sigma(f_*\mu)\leqslant v$

\item \eqref{eq: another} holds for any R satisfying \eqref{eq: rdefinition}.
\end{enumerate}
\end{thm}
\begin{proof}
Suppose the second statement holds then the set in the first statement has full measure hence
is nonempty.

If the third statement holds previous discussion shows that
$\eqref{eq: another}\Leftrightarrow \eqref{eq: equicritical} \Leftrightarrow \eqref{eq: critical}$. We may apply
Proposition \ref{prop: proposition} to get the second statement.

If the third statement fails to hold, then no ball B intersecting supp $\mu$
satisfies \eqref{eq: critical}. By Lemma \ref{lem: negation} $f(B\bigcap\textrm{supp
}\mu) \subset \Sigma_u$ for some $u>v$. This would undermine the first statement.
\end{proof}
From Theorem \ref{thm: anothermain} we see that $\sigma(L)\leqslant\inf\{\su| \y \in L\}$ because the
first statement implies the second one.  $\sigma(L)\geqslant\inf\{\su| \y \in L\}$ is apparent by definition.

$\sigma(L)$ is inherited by
its nondegenerate submanifolds because  nondegeneracy implies ($f,\mu$)-goodness and ($f,\mu$)-nonplanarity by previous 
conceptual discussions.
 Therefore \\$\sigma(L)=\sigma(M)=\inf\{\su| \y \in L\}=\inf\{\su| \y \in M\}$  and Theorem \ref{thm: mainthm} is established.

Besides, Theorem \ref{thm: anothermain} establishes that
\begin{equation}\label{eq: computingsigmal}
\sigma(L)=\sup\{v\;| \;\eqref{eq: critical} \textrm{ does not hold}\}
\end{equation}

More significantly, it yields a more general result than Theorem \ref{thm: mainthm}.
\begin{thm}\label{thm: yetanothermain}
Let $\mu$ be a Federer measure on a Besicovitch metric space X, L an
affine subspace of $\R^n$, and let $f: X\rightarrow L$ be a
continuous map such that  ($f,\mu$) is  good  and  nonplanar in $L$ then
\begin{equation}
\sigma(f_*\mu)=\sigma(L)=\inf\{\su| \y \in L\}=\inf\{\sigma(f(x))| x \in \textrm{supp }\mu \}
\end{equation}
\end{thm}
\cite{KLW}, for instance,  studied 'absolutely decaying and Federer' measures and proved that if $\mu$ is
absolutely decaying and Federer, and $f$ is nondenegerate at $\mu-a.e.$ points of $\R^d$, then
$(f,\mu)$ is good and nonplanar. Theorem \ref{thm: yetanothermain} is applicable to such generalized situations.

\medskip

To make all this more explicit, first note that for an
affine subspace  $L$ of dimension s  matrix $A$ as described in \eqref{eq: par} can
be read from matrix $R$ as in \eqref{eq: rdefinition} and vice versa,since
\begin{equation}
R=
 \left(  \begin{array}{cc}
 A & I_{s+1}
\end{array} \right)
\end{equation}
Set $\sigma_j(A)(1\leqslant j\leqslant n+1)=$
\begin{equation}\label{eq: sigmaofa}
\textrm{sup}\{v\big|\exists \w \in S_{n+1,j} \textrm{
with arbitrarily large }\|\pi(\w)\| \textrm{ and
}\|RC(\w)\|< \|\pi(\w)\|^{-\frac{jv}{v+1-jv}}\}
\end{equation}
 and we derive

\begin{cor}\label{cor: sigmal}
If L is an $s$ dimensional affine subspace of $\R^n$ parameterized
by \eqref{eq: par}, then
\begin{equation}\label{eq: corsigmal}
 \sigma(L)=\textrm{max
 }\{1/n,\sigma_1(A),\sigma_2(A),\ldots,\sigma_n(A)\}
\end{equation}
\end{cor}
\begin{proof}
Note $\eqref{eq: another}\Leftrightarrow \eqref{eq: equicritical} \Leftrightarrow \eqref{eq: critical}$ then
apply \eqref{eq: computingsigmal}.

 It remains to elucidate $\sigma_{n+1}(A)$.
 $\bigwedge^{n+1}(\R^{n+1})$ is spanned by a single element  $
\mathbf{e}_1\wedge\dots\wedge\mathbf{e}_{n+1}$, hence $\|g_tu_{\tilde{f}} \w\|_{B, \mu}$ has a positive lower bound.
\eqref{eq: critical} is always met as long as $v>1/n$ for $j=n+1$, and we replace $\sigma_{n+1}(A)$ with $1/n$ in \eqref{eq: corsigmal}.
\end{proof}

\section{Several examples}
 Corollary \ref{cor: sigmal} will prove to be effective for deriving explicit formulas of $\sigma(L)$. First we have
\begin{thm}\label{thm: sigmana}
$\sigma_n(A)=\dfrac{\omega(A)}{n+(n-1)\omega(A)}$
\end{thm}
\begin{proof}
For $\w \in S_{n+1,n}$,
\begin{equation}
\w=\sum_{j=1}^{n+1}x_j\mathbf{e}_{T\smallsetminus \{j\}}
\end{equation}
 where
$T=\{1,2,\ldots,n,n+1\}$ and $x_j \in \Z$. \\
Therefore
\begin{equation}
C_j(\w)=x_j\mathbf{e}_{T\smallsetminus \{n+1\}}\qquad 1\leqslant j\leqslant n+1
\end{equation}
$\pi(\w)=\sum_{j=1}^{n}x_j\mathbf{e}_{T\smallsetminus \{j\}}$
\qquad $\|\pi(\w)\|=\sqrt{x_1^2+x_2^2+\ldots+x_n^2}$\\
\begin{displaymath}
\left\|RC(\w)\right\|=\left\|(A\;\; I_{s+1})\left(
                \begin{array}{c}
                  x_1 \\
                  x_2 \\
                  \vdots \\

                  x_{n+1} \\
                \end{array}
              \right)\right\|=\left\| \left(
                                     \begin{array}{c}
                                       x_{n-s+1} \\
                                       \vdots \\
                                       x_n \\
                                       x_{n+1} \\
                                     \end{array}
                                   \right)+A\left(
                                              \begin{array}{c}
                                                x_1 \\
                                                x_2 \\
                                                \vdots \\
                                                x_{n-s} \\
                                              \end{array}
                                            \right)
                \right\|
\end{displaymath}
$\sigma_n(A)$  is consequently equal to the supremum of $v$ such that there exists
$\x=(x_1,\ldots,x_{n+1}) \in Z^{n+1}$ with arbitrarily
large $\sqrt{x_1^2+\ldots+x_n^2}$  and
\begin{equation}\label{eq: matrixsigma}
\left\| \left(
                                     \begin{array}{c}
                                       x_{n-s+1} \\
                                       \vdots \\
                                       x_n \\
                                       x_{n+1} \\
                                     \end{array}
                                   \right)+A\left(
                                              \begin{array}{c}
                                                x_{1} \\
                                                x_{2} \\
                                                \vdots \\
                                                x_{n-s} \\
                                              \end{array}
                                            \right)
                \right\|< \big(\sqrt{x_1^2+\ldots+x_n^2}\big)^{-\frac{nv}{v+1-nv}}
\end{equation}

In order for $\|RC(\w)\|$ to be sufficiently small, left side of the inequality of \eqref{eq: matrixsigma} has to be less than 1, hence
$(x_{n-s+1},\ldots,x_n,x_{n+1}) \in \Z^{s+1}$ are determined by
$(x_{1},x_{2},\ldots,x_{n-s}) \in \Z^{n-s}$.
Up to some constant
\begin{equation}
\sqrt{x_1^2+\ldots+x_n^2}\asymp\sqrt{x_1^2+\ldots+x_{n-s}^2}
\end{equation}
Now that the  definitions of
$\sigma_n(A)$ and $\omega(A)$ differ only by the exponents,
 we conclude that
\begin{equation}
\dfrac{n\sigma_n(A)}{1+(n-1)\sigma_n(A)}=\omega(A)\Rightarrow
\sigma_n(A)=\dfrac{\omega(A)}{n+(n-1)\omega(A)}
\end{equation}
\end{proof}
\medskip

Next we set out to prove Theorem \ref{thm: hyperplane} noting that we will be able to
eliminate all $\sigma_j(A)$ for $j<n$ and only $\sigma_n(A)$ matters.

\begin{proof} [Proof of Theorem \ref{thm: hyperplane}]

The aforementioned matrix $R$ satisfies\\
$(a_1x_1+\ldots+a_{n-1}x_{n-1}+a_n,x_1,\ldots,x_{n-1},1)=(x_1,\ldots,x_{n-1},1)R$.
Therefore
\begin{displaymath}
R=\left(
    \begin{array}{cc}
      A & I_n \\
    \end{array}
  \right)
\end{displaymath}
\begin{displaymath}
A=
 \left(  \begin{array}{c}
 a_1\\
 a_2\\
 \vdots\\
 a_n\\
\end{array} \right)
\end{displaymath}
\begin{equation}\label{eq: rcwhyperplane}
\|RC(\w)\|=\max\{
\|C_2(\w)+a_1C_1(\w)\|,\ldots,\|C_{n+1}(\w)+a_nC_1(\w)\|\}
\end{equation}
We claim that $\sigma_j(A)$ $(1\leqslant j \leqslant n-1$ remain
zero. To see this, consider $\w \in S_{n+1,j}$: as the norm of
$\pi(\w)$ is tending to $\infty$ it suffices to show that
$\|RC(\w)\|>\epsilon$ for some $\epsilon>0$ for large $\|\pi(\w)\|$.

Suppose not and assume first that $a_1,a_2,\ldots,a_{n-1}$ are all
nonzero. Recall
\begin{equation*}
\pi(\w)=\sum_{\substack{ \textrm{\#}I=j \\ {n+1} \in I
}}\langle\mathbf{e}_I, \w\rangle\mathbf{e}_I \qquad
\end{equation*}
\begin{displaymath}
C_j(\w)=\sum_{\substack{ j \in I \\ I \subset
\{{1,2,\cdots,n}\}}}\langle\mathbf{e}_{I \backslash{\{j\}} \cup
{\{n+1\}}}, \w\rangle\mathbf{e}_I  \qquad 1\leqslant j\leqslant n
\end{displaymath}
\begin{equation*}
 C_{n+1}(\w)=\sum_{ I \subset
\{{1,2,\cdots,n}\}}\langle\mathbf{e}_{I}, \w\rangle\mathbf{e}_I
\end{equation*}

Each term in $C_1(\w)$ is of the form $\langle\mathbf{e}_{I
\backslash{\{1\}} \cup {\{n+1\}}}, \w\rangle\mathbf{e}_I,1 \in I $.
For an arbitrary one, since  the index set $I$ has $i<n$ elements, $\exists k>1$ such that $k
\notin I$.

Consider $\|C_k(\w)+a_{k-1}C_1(\w)\|$(it cannot have a
positive lower bound by assumption), $\|C_k(\w)+a_{k-1}C_1(\w)\|\geqslant \|
a_{k-1}\langle\mathbf{e}_{I \backslash{\{1\}} \cup {\{n+1\}}},
\w\rangle\mathbf{e}_I\|$, \\therefore $\langle\mathbf{e}_{I \backslash{\{1\}}
\cup {\{n+1\}}}\w\rangle\mathbf{e}_I$ must be equal to zero.
Consequently  \\$C_1(\w)=0$. $C_i(\w)=0$ ($2 \leqslant i \leqslant n$) are forced to be zero by
\eqref{eq: rcwhyperplane}.
This contradicts the fact that $\pi(\w)$ is nonzero.

For arbitrary $A^t=(a_1,\ldots,a_t,0,\ldots,0,a_n)$ first note that according to  \eqref{eq: rcwhyperplane}
we have  $C_i(\w)=0$ ($t+1 \leqslant i \leqslant n$) or $\|RC(\w)\|$ cannot be arbitrarily small.
\begin{equation}\label{eq: anotherrcwhyperplane}
\|RC(\w)\|\geqslant\max\{
\|C_2(\w)+a_1C_1(\w)\|,\ldots,\|C_{t+1}(\w)+a_tC_1(\w)\|\}
\end{equation}

Employing previous analysis on \eqref{eq: anotherrcwhyperplane} shows that $C_i(\w)=0$ ($1 \leqslant i \leqslant t$) hence $\pi(\w)$ has to be
zero.

Therefore $\sigma_j(A)=0$ for $j<n$. Combining  Corollary \ref{cor: sigmal} and Theorem \ref{thm: sigmana},
Theorem \ref{thm: hyperplane} is established.
\end{proof}

\begin{remark}
When $A^t$ is of a special form $(0,0,\ldots,0,a)$ our conclusion
coincides with Satz 3 of \cite{Jarnik}. The latter proved this special case with
elementary method. This method, however, is not easily adjustable to
more general situations.
\end{remark}

\medskip

In the next theorem  we will study other subspaces than hyperplanes which highlight  $\sigma_j(A)$ with  $j<n$.

\begin{thm}
Consider a line ($L$)  $\subset \R^3$ that passes through the origin parameterized by
$x\rightarrow
(ax,bx,x)$. Set $\y=(a,b) \in \R^2$, then
\begin{equation}\label{eq: linesigma}
 \sigma(L)=\max\{1/3,\dfrac{\sigma(\y)}{2+\sigma(\y)},\dfrac{\omega(\y)}{3+2\omega(\y)}\}
\end{equation}
\end{thm}

\begin{proof}
Because $A=\left(
             \begin{array}{cc}
               a & b \\
               0 & 0 \\
             \end{array}
           \right)
$ in this case, by Corollary \ref{cor: sigmal} and Theorem \ref{thm: sigmana} we only need to prove
$\sigma_1(A)=0$ and $\sigma_2(A)=\dfrac{\sigma(\y)}{2+\sigma(\y)}$.

For $\w \in S_{4,1}$, $\w$ can be expressed as $x_1\mathbf{e}_1+\ldots+x_4\mathbf{e}_4$ with $x_i \in \Z$.
\begin{equation*}
\pi(\w)=x_4\mathbf{e}_4
\end{equation*}
\begin{equation*}
C_j(\w)=x_4\mathbf{e}_j \; 1\leqslant j\leqslant 3 \qquad  C_{4}(\w)=x_1\mathbf{e}_1+x_2\mathbf{e}_2+x_3\mathbf{e}_3
\end{equation*}
\begin{displaymath}
\Big\|RC(\w)\Big\|=\Big\|(A\;\; I_{2})C(\w)\Big\|\geqslant|x_4|=\|\pi(\w)\|
\end{displaymath}
By \eqref{eq: sigmaofa}, $\sigma_1(A)$ is zero as $\|RC(\w)\|$ becomes unbounded. Moreover for
 \begin{equation}
\w \in S_{4,2}=\sum_{1\leqslant i<j\leqslant 4} x_{i,j}\mathbf{e}_i\wedge \mathbf{e}_j,\quad x_{i,j}\in \Z
\end{equation}
  We have by \eqref{eq: sigmaofa} $\sigma_2(A)$  equal to the supremum of $v$ such that there exists
$(p_1,p_2,q) \in Z^3$ with arbitrarily
large $|q|$  and
\begin{equation}
\big\|\begin{array}{c}
  qa+p_1 \\
  qb+p_2
\end{array}\big\|
<|q|^{-\frac{2v}{v+1-2v}}
\end{equation}
Hence $\sigma_2(A)=\dfrac{\sigma(\y)}{2+\sigma(\y)}$ as desired.

\end{proof}

\section{Further remarks}

We study one low dimension example to see some distinction between $\sigma(L)$ and $\omega(L)$.
 Let $L=\{(x,a)|\;x \in \R\}$  with $\sigma(a)>2$. From definition we know at once that $\omega(x,a)\geqslant\sigma(a)\; \forall x$, hence
\begin{equation}\label{omegalgreater}
 \omega(L)\geqslant\sigma(a)
\end{equation}

\cite{exponent} by using dynamics showed that the lower bound was actually attained
\begin{equation}\label{omegal}
\omega(L)=\sigma(a)
\end{equation}

From \eqref{eq: transference1} and \eqref{omegal} we derive that $\sigma(L)\geqslant \dfrac{1}{1+2/\sigma(a)}$. However there seems to be
no way to know the exact value of $\sigma(L)$ simply from results of $\omega(L)$.

On the other hand from Theorem \ref{thm: hyperplane} we derive that
 \begin{equation}\label{sigmavalue}
\sigma(L)=\dfrac{1}{1+2/\sigma(a)}
\end{equation}

It turns out \eqref{sigmavalue} suffices to generate the exact value of $\omega(L)$ if we consider the following
argument:
by \eqref{eq: transference1}
\begin{equation}\label{transgreater}
\sigma(L)\geqslant\dfrac{1}{1+2/\omega(L)}
\end{equation}
By \eqref{sigmavalue} and \eqref{omegalgreater} and the fact that $f(x)=\dfrac{1}{1+2/x}$ is increasing we have
\begin{equation}\label{transmaller}
\sigma(L)\leqslant\dfrac{1}{1+2/\omega(L)}
\end{equation}
\eqref{transgreater} and \eqref{transmaller} show that $\sigma(L)=\dfrac{1}{1+2/\omega(L)}$. Comparing this with \eqref{sigmavalue}
we see that $\omega(L)=\sigma(a)$ as desired.

\bigskip

\textbf{Acknowledgement.} The author is grateful to Professor
Kleinbock for helpful discussions.

\end{document}